\documentclass[12pt,reqno]{amsart}
\usepackage[utf8]{inputenc}
\usepackage[T1]{fontenc}
\usepackage{amsmath}
\usepackage{amsthm}
\usepackage{amssymb}
\usepackage[abbrev]{amsrefs}
\usepackage{mathrsfs}
\usepackage[dvipsnames]{xcolor}
\usepackage{bm}
\usepackage{enumitem}
\usepackage{hyperref}
\usepackage{graphicx}
\usepackage[all]{xy}
\AtBeginDocument{\def\MR#1{}}
\makeatletter
\@namedef{subjclassname@2020}{%
\textup{2020} Mathematics Subject Classification}
\makeatother
\numberwithin{equation}{section}
\allowdisplaybreaks
\newtheorem{theorem}{Theorem}[section]
\newtheorem{proposition}[theorem]{Proposition}
\newtheorem{lemma}[theorem]{Lemma}
\newtheorem{corollary}[theorem]{Corollary}
\theoremstyle{remark}

\newcommand{\littletaller}{\mathchoice{\vphantom{\big|}}{}{}{}}
\newcommand\restr[2]{{\left.\kern-\nulldelimiterspace #1 \littletaller \right|_{#2}}}
\newcommand{\enpar}[1]{\left(#1\right)}
\newcommand{\enbrace}[1]{\left\lbrace#1\right\rbrace}
\newcommand{\abs}[1]{\left\lvert#1\right\rvert}
\newcommand{\norm}[1]{\left\lVert#1\right\rVert}
\newcommand{\NN}{\mathbb{N}}
\newcommand{\RR}{\mathbb{R}}
\newcommand{\F}{\mathcal{F}}

\newcommand{\Acal}{\mathcal{A}}
\newcommand{\Mcal}{\mathcal{M}}
\newcommand{\Id}{\ensuremath{{\rm{Id}}}}

\newcommand{\spn}{\operatorname{span}}

\title[Isometries without linear shadows]{Isometries without linear shadows: failure of Figiel's theorem beyond local convexity}
\author[F. Albiac]{Fernando Albiac}
\address{Institute for Advanced Materials and Mathematics (INAMAT$^{2}$) and Department of Mathematics, Statistics and Computer Sciences\\ Public University of Navarre\\
Campus de Arrosad\'{\i}a, 31006 Pamplona\\ Spain}
\email{fernando.albiac@unavarra.es}
\author[J. L. Ansorena]{Jos\'e L. Ansorena}
\address{Department of Mathematics and Computer Sciences\\
Universidad de La Rioja\\
Logro\~no\\
26004 Spain}
\email{joseluis.ansorena@unirioja.es}
\author[M. C\'uth]{Marek C\'uth}
\address{Faculty of Mathematics and Physics, Department of Mathematical Analysis\\
Charles University\\
186 75 Praha 8\\
Czech Republic}
\email{cuth@karlin.mff.cuni.cz}
\date{}
\subjclass[2020]{46A16, 46B04 (Primary), 46B85, 51F30 (Secondary)}
\keywords{$p$-Banach spaces, isometries, Lipschitz-free $p$-spaces, Figiel's theorem}
\thanks{F.\@ Albiac acknowledges the support of the Spanish Ministry for Science, Innovation, and Universities under Grant PID2025-167660NB-I00 funded by MICIU/AEI/10.13039/501100011033 and ERDF/EU. F.\@ Albiac and J.\@ L.\@ Ansorena acknowledge the support of the Spanish Ministry for Science, Innovation, and Universities under Grant PGC2018-095366-B-I00 for \emph{Functional Analysis Techniques in Approximation Theory and Applications (TAFPAA)}}
\begin{document}
\begin{abstract}
A classical theorem of Figiel asserts that every isometric embedding $\varphi\colon X\to Y$ between real Banach spaces, with $\varphi(0)=0$, admits a linear contractive left inverse on the closed linear span of $\varphi(X)$. We show that this phenomenon breaks down completely once local convexity is lost. More precisely, for every $0<p<1$ and every nonzero real Banach space $X$, we construct a $p$-Banach space $Y$ and an isometric embedding $\varphi\colon X\to Y$ such that $\varphi(0)=0$ and the closed linear span of $\varphi(X)$ is the entire space $Y$, but there is no linear map, bounded or not, on the linear span of $\varphi(X)$ which is a left inverse of $\varphi$. The construction is carried out using the $p$-Banach counterpart of Lipschitz-free spaces. In fact, $Y$ may be chosen isomorphic to the Lipschitz-free $p$-space over $X$. The main ingredient is a quantitative estimate controlling the distance to elementary molecules in Lipschitz-free $p$-spaces, a phenomenon specific to the range $0<p<1$. As a further consequence, we show that a real Banach space has the isometric $p$-lifting property if and only if its dimension is at most one.
\end{abstract}
\maketitle
\section{Introduction}\noindent
Every Banach space is, by its nature, both a linear and a metric space. One of the fundamental questions of nonlinear functional analysis is to understand how much of its linear structure is already encoded in the metric. Thus, if two Banach spaces are Lipschitz isomorphic, or uniformly homeomorphic, to what extent must their linear structures resemble one another? Similar questions arise for nonlinear embeddings, where one asks what linear information, if any, can be recovered from the metric behavior of the embedding. There are several versions of this problem, depending on how much metric information is assumed to be preserved. At one end of the scale one may ask only for a homeomorphism, while at the other one requires the exact preservation of distances. Between these two extremes lie the Lipschitz and uniform categories which have played a central role in the nonlinear geometry of Banach spaces.

The situation becomes considerably less transparent once local convexity is lost. Although the nonlinear theory of Banach spaces has developed into a substantial subject, comparatively little is known about the nonlinear geometry of quasi-Banach spaces and, more generally, of $\textsf{F}$-spaces. The classical Lipschitz isomorphism problem asks whether two separable Banach spaces which are Lipschitz isomorphic must be linearly isomorphic. This problem remains open in full generality, but the Banach-space theory is by no means devoid of positive results. A major source of such results is the differentiation theory of Lipschitz maps associated with Aronszajn \cite{Aronszajn1976}, Christensen \cite{Christensen1972} and Mankiewicz \cite{Mankiewicz1973}, and its subsequent use by Heinrich and Mankiewicz \cite{HM1982} in nonlinear classification. The basic principle is simple: differentiability of a Lipschitz map at a point produces a bounded linear operator which retains part of the metric information carried by the original map. This has become one of the fundamental linearization devices in nonlinear Banach space theory (see, for instance, \cites{BenLin2000,AlbiacKalton2016}).

Outside local convexity the picture is, in a sense, almost the reverse. In \cite{AlbiacKalton2009}, the first-named author and Kalton constructed separable quasi-Banach spaces which are Lipschitz isomorphic but not linearly isomorphic, so the corresponding Lipschitz isomorphism problem has a negative answer in general. On the other hand, there is no comparable general positive linearization theory. Recent work \cite{AlbiacAnsorenaWald2027} gives a rather striking reason for this. If a quasi-Banach space $X$ is not locally convex, then there is a Lipschitz map $ f\colon\RR\to X$ which is nowhere differentiable. In fact, the apparently much weaker requirement that every Lipschitz map $\RR\to X$ have merely one point of differentiability already forces $X$ to be locally convex and to have the Radon--Nikod\'ym property. Thus one of the principal linearization tools available in nonlinear Banach space theory disappears altogether in the genuinely nonlocally convex setting.

The counterexamples in \cite{AlbiacKalton2009} were based on a new construction, a $p$-analogue of the Arens--Eells space. This idea was later developed systematically into the theory of Lipschitz-free $p$-spaces over $p$-metric spaces \cite{AACD2020}. These spaces will also provide the natural framework for the counterexamples constructed in the present paper.

At the isometric end of the scale, however, the classical picture is far more rigid. The theorem of Mazur and Ulam \cite{MazurUlam1932} says that if $X$ and $Y$ are real Banach spaces and $U\colon X\to Y$ is a surjective isometry with $U(0)=0$, then $U$ is linear. Thus, in the real Banach-space setting, an onto isometry already remembers the linear structure. Rather surprisingly, local convexity is not needed for this conclusion: Rolewicz \cite{Rolewicz1968} proved the corresponding Mazur--Ulam theorem for real quasi-Banach spaces (cf.~\cite{Albiac2008}*{Corollary~2.3}).

Surjectivity, however, leaves open what happens for isometric embeddings. In 1968, Figiel showed that a remarkable part of the Mazur--Ulam rigidity survives even without it. We recall his theorem in the form that is most convenient here. Throughout this paper, the symbol $[Z]$ represents the closed linear span of a subset $Z$ of a quasi-Banach space.

\begin{theorem}[see \cite{Figiel1968}]\label{thm:Figiel}
Let $X$ and $Y$ be real Banach spaces and let $\varphi\colon X\to Y$ be an isometric embedding with $\varphi(0)=0$. Then, there is a linear contraction $T\colon [\varphi(X)] \to X$ such that $ T\circ \varphi=\Id_X$.
\end{theorem}

Thus an isometric copy of a Banach space inside another Banach space, although it need not itself be a linear subspace, always carries a linear ``shadow'' of the original space. The Mazur--Ulam theorem follows at once: if $\varphi$ is onto, the operator $T$ in Figiel's theorem is simply $\varphi^{-1}$, and hence $\varphi$ is linear. More generally, Theorem~\ref{thm:Figiel} has become a useful linearization principle in nonlinear Banach space theory. In particular, it enters naturally into the theory of Lipschitz-free Banach spaces and the theorem of Godefroy and Kalton that a separable Banach space which embeds isometrically into a Banach space $Y$ also embeds linearly and isometrically into $Y$ (see \cite{GodefroyKalton2003}).

There is, however, a strong local-convexity ingredient hidden in this second linearization principle as well. Figiel's proof is essentially dual in nature. One way to phrase its key step is that, for every $x^*\in X^*$, one finds $y^*\in Y^*$ with $y^*(\varphi(x))=x^*(x)$ for all $x\in X$, and the Figiel operator is then recovered from these functionals (see, e.g., \cite{AlbiacKalton2016}*{proof of Theorem 14.4.10}). Thus separation by continuous linear functionals, and ultimately the Hahn--Banach principle, are built into the proof. In a nonlocally convex quasi-Banach space the continuous dual may be too small for such an argument, or even trivial, and no obvious substitute is available.

This naturally raises the question whether Figiel's theorem itself survives without local convexity. Kalton posed this question to the first-named author in 2003, when the nonlinear geometry of quasi-Banach spaces was only beginning to take shape. The failure of the usual proof is, of course, not evidence that the theorem itself must fail: many Banach-space results whose classical proofs rely on separation or duality continue to hold in the quasi-Banach setting for quite different reasons. Here, however, the theorem itself breaks down.

To state our result, recall that for $0<p\le 1$ a $p$-Banach space is a complete quasi-normed space $X$ whose quasi-norm satisfies the $p$-triangle law
\[
\norm{x+y}^p\leq \norm{x}^p+\norm{y}^p, \qquad x,y\in X.
\]
Every Banach space is of course a $p$-Banach space for every $0<p\le 1$. Our main theorem shows that Figiel's conclusion fails in the strongest possible sense, even when the domain of the isometric embedding is kept within the Banach category.

\begin{theorem}\label{thm:main}
Let $0<p<1$ and let $X$ be a nonzero real Banach space. Then there are a real $p$-Banach space $Y$ and an isometric embedding
\[
\varphi\colon X\to Y
\]
such that $\varphi(0)=0$ and $[\varphi(X)]=Y$, but there is no linear map
\[
T\colon \spn(\varphi(X)) \to X
\]
such that $T\circ\varphi=\Id_X$. In fact, $Y$ can be chosen to be isomorphic to the Lipschitz-free $p$-space over $X$.
\end{theorem}

Two features of Theorem~\ref{thm:main} are worth emphasizing. First, the failure is already algebraic: there is no linear left inverse even on the linear span of the range of $\varphi$, without any boundedness requirement. Second, the counterexample does not arise from a complicated extension: the target space is obtained from $X$ by standard functional constructions. In fact, a quotient of $\F_p(X)$ by a one-dimensional complemented subspace is enough to destroy Figiel's conclusion completely.

We next recall the few facts about Lipschitz-free $p$-spaces that are
needed to formulate the construction. We take the convention that any $p$-Banach $X$ is a pointed $p$-metric space with base point the zero vector and $p$-distance
\[
(x,y) \mapsto d_X(x,y):=\norm{x-y}, \qquad x,y\in X.
\]
In general, for any $p$-metric space $(\Mcal, d_{\Mcal})$ with a distinguished point $0\in \Mcal$, there is a $p$-Banach space $\F_p(\Mcal)$ and a canonical isometry
\[
\delta_{p,\Mcal}\colon \Mcal \to \F_p(\Mcal)
\]
that satisfy the natural universal property. Namely, for every $p$-Banach space $Z$ and every Lipschitz contraction $f\colon \Mcal \to Z$ with $f(0)=0$ there is a linear contraction $E[f;p]\colon \F_p(\Mcal) \to Z$ such that
\[
E[f;p] \circ \delta_{p,\Mcal} =f
\]
(see \cites{AlbiacKalton2009,AACD2020}). The family
\[
(\delta_{p,\Mcal}(x))_{x\in \Mcal\setminus\{0\}}
\]
is linearly independent, and its closed linear span is the entire space $\F_p(\Mcal)$. Furthermore, series expansions in terms of norm-one vectors of the form
\begin{equation}\label{eq:atom}
\mu_{p,\Mcal}(x,y)= \frac{\delta_{p,\Mcal}(x)-\delta_{p,\Mcal}(y)}{d_{\Mcal}(x,y)}, \qquad x,y \in \Mcal, \, x\neq y,
\end{equation}
encode the $p$-norm of $\F_p(\Mcal)$. Specifically, if we set
\begin{equation}\label{eq:atomset}
\Acal_{p,\Mcal}=\enbrace{\mu_{p,\Mcal}(x,y) \colon x,y\in \Mcal, \, x\not=y},
\end{equation}
then, for any $\mu\in \F_p(\Mcal)$,
\begin{equation}\label{eq:atomic-formula}
\norm{\mu}=\inf \enbrace{ \enpar{\sum_{j=1}^\infty \abs{a_j}^p}^{1/p}
\colon a_j\in\RR, \mu_j\in \Acal_{p,\Mcal}, \mu=\sum_{j=1}^\infty a_j\, \mu_j}
\end{equation}
(see \cite{AACD2020}*{Corollary 4.12}). We refer to the elements of $\Acal_{p,\Mcal}$ as
\emph{elementary molecules}.

If $X$ is a $p$-Banach space, we regard the identity map $\Id_X$ as a map from a $p$-metric space into a $p$-Banach space. So, the universal property of Lipschitz free $p$-spaces yields a linear contraction
\begin{equation}\label{eq:bary}
\beta_{p,X}:=E[\Id_X;p] \colon \F_p(X) \to X.
\end{equation}
The onto map $\beta_{p,X}$ will be called the $p$-\emph{barycentric map} of $X$.

The following elementary observation isolates the mechanism behind the counterexample.

\begin{proposition}\label{prop:quotient-obstruction}
Let $X$ be a nonzero $p$-Banach space, $0<p\le 1$. Suppose that $\rho\in\spn(\delta_{p,X}(X))$ satisfies $\beta_{p,X}(\rho)\neq0$ and $d_{\F_p(X)}(\mu,[\rho])=1$ for every $\mu\in\Acal_{p,X}$. Set $Y=\F_p(X)/[\rho]$, let $q\colon\F_p(X)\to Y$ be the canonical quotient map, and put $\varphi=q\circ\delta_{p,X}$. Then $\varphi\colon X\to Y$ is an isometric embedding with $\varphi(0)=0$ and $[\varphi(X)]=Y$, but there is no linear map
\[
T\colon\spn(\varphi(X))\to X
\]
such that $T\circ\varphi=\Id_X$.
\end{proposition}

\begin{proof}
For distinct $x,y\in X$, the definition of the quotient $p$-norm gives
\begin{align*}
\norm{\varphi(x)-\varphi(y)}
&=\norm{x-y}\,
\norm{q(\mu_{p,X}(x,y))}\\
&=\norm{x-y}\,
d_{\F_p(X)}(\mu_{p,X}(x,y),[\rho])
=\norm{x-y}.
\end{align*}
Thus $\varphi$ is isometric, and clearly $\varphi(0)=0$. Moreover, since $\F_p(X)=[\delta_{p,X}(X)]$ and $q$ is continuous and onto,
\[
Y=q(\F_p(X))
\subseteq[q(\delta_{p,X}(X))]
=[\varphi(X)]
\subseteq Y.
\]
Hence $[\varphi(X)]=Y$.

Finally, write $\rho=\sum_{j=1}^n a_j\delta_{p,X}(x_j)$. Since $q(\rho)=0$,
\[
\sum_{j=1}^n a_j\varphi(x_j)=0,
\]
whereas
\[
\sum_{j=1}^n a_jx_j=\beta_{p,X}(\rho)\neq0.
\]
Thus a linear map $T\colon\spn(\varphi(X))\to X$ satisfying
$T\circ\varphi=\Id_X$ would give
\[
0=T\enpar{\sum_{j=1}^n a_j\varphi(x_j)}
=\sum_{j=1}^n a_jx_j
=\beta_{p,X}(\rho),
\]
a contradiction.
\end{proof}

Proposition~\ref{prop:quotient-obstruction} reduces the main part of the proof of Theorem~\ref{thm:main} to the construction of a finitely supported vector $\rho\in\F_p(X)$ such that $\beta_{p,X}(\rho)\neq0$ and $d_{\F_p(X)}(\mu,[\rho])=1$ for every $\mu\in\Acal_{p,X}$. In Section~\ref{sec:BVCL} we will see a concentration estimate that will reduce this task to a more concrete one: it suffices to find a finitely supported vector in $\ker\beta_{p,X}$ which is uniformly separated from the one-dimensional spaces $[\mu-\nu]$, where $\mu,\nu\in\Acal_{p,X}$. In Section~\ref{sec:transverse} we construct a vector $\gamma$ with the required separation property. Finally, in Section~\ref{sec:proof} we choose a small perturbation $\rho$ of $\gamma$ with $\beta_{p,X}(\rho)\neq0$ and show that the resulting quotient $\F_p(X)/[\rho]$ is isomorphic to $\F_p(X)$.

The concentration phenomenon to which we alluded above will have a further consequence that we will explore in Section~\ref{sec:isometric-p-liftings} by characterizing the Banach spaces with the isometric $p$-lifting property. Recall that, following the terminology introduced by Godefroy and Kalton \cite{GodefroyKalton2003} in the Banach case, and extended to the case $p\le 1$ in \cite{AlbiacKalton2009}, we say that a $p$-Banach space $X$ has the \emph{isometric $p$-lifting property} if the $p$-barycentric map $\beta_{p,X}$ admits a contractive linear lifting, that is, there is a contractive linear map
\[
R\colon X\to\F_p(X)
\]
such that $\beta_{p,X}\circ R=\Id_X$. While every separable Banach space has the isometric $1$-lifting property (see \cite{GodefroyKalton2003}*{Theorem 3.1}), we prove that a Banach space has the isometric $p$-lifting property for $p<1$ only if it is trivial or one-dimensional. This provides a further illustration of the sharp change which takes place at the locally convex endpoint $p=1$.

Our paper requires little background beyond the specific Lipschitz-free $p$-space facts needed in the proof of the main theorem, and the above paragraphs already say exactly which ingredients are used: the canonical embedding, density, universal property and elementary molecule description, all of which can be found in \cite{AACD2020}. Aside from that, we use standard notation and terminology from Banach space theory, as in \cite{AlbiacKalton2016}. For general background on quasi-Banach spaces and $\textsf{F}$-spaces we refer to \cite{KPR1984}. Throughout the paper all vector spaces are real, and an isometry is not assumed to be linear.
\section{Quantitative proximity to elementary molecules}\label{sec:BVCL}\noindent
Fix $0<p\le 1$ and a nonzero real $p$-Banach space $X$. Let $S_X$ denote the unit sphere of $X$. By \eqref{eq:atom} and \eqref{eq:bary},
\begin{equation}\label{tag2.1}
\beta_{p,X}\enpar{\mu_{p,X}(x,y)}
=\frac{x-y}{\norm{x-y}} \in S_X, \qquad x,y\in X, \, x\not=y.
\end{equation}

The aim of this section is to estimate quantitatively the distance from $\mu\in\F_p(X)$ to $\Acal_{p,X}$ under the assumption
that $\beta_{p,X}(\mu)\in S_X$. We first record an immediate lower bound. Since $\beta_{p,X}$ is contractive, $\|\mu\|\geq1$. Hence, for every $\nu\in\Acal_{p,X}$, the reverse $p$-triangle inequality gives
\[
\norm{\mu-\nu}^p \geq \norm{\mu}^p-\norm{\nu}^p = \norm{\mu}^p-1.
\]
Consequently,
\begin{equation*}
d_{\F_p(X)}\enpar{\mu,\Acal_{p,X}}^p
\geq \norm{\mu}^p-1.
\end{equation*}

When $X$ is a Banach space and $0<p<1$, the next lemma gives a converse estimate, up to a constant depending only on $p$. In particular, within the fiber $\{u\in \F_p(X): \|\beta_{p,X}(u)\|=1\}$, points whose norm is close to the smallest possible value $1$ must be close to $\Acal_{p,X}$. This concentration principle will drive the construction in the following sections. Its proof uses the elementary inequality
\begin{equation}\label{eq:sb}
\norm{\alpha}_1\le \max\enbrace{ \varepsilon\norm{\alpha}_p, \varepsilon^{-p/(1-p)} \norm{\alpha}_\infty}, \quad \varepsilon\in(0,1], \, \alpha\in\ell_p
\end{equation}
(see \cite{AAW2021}*{Lemma 3.2}). For $\varepsilon=1$, inequality \eqref{eq:sb} reduces to the contractive inclusion of $\ell_p$ into $\ell_1$, namely,
\begin{equation}\label{eq:lpl1}
\norm{\alpha}_1\le \norm{\alpha}_p, \qquad \alpha\in\ell_p.
\end{equation}
We shall also use the symmetry of $\Acal_{p,X}$, which follows immediately from the fact that
\[
\mu_{p,X}(y,x)=-\mu_{p,X}(x,y),
\qquad x,y\in X,\quad x\neq y.
\]

\begin{lemma}\label{lem:concentration}
Let $0<p<1$ and let $X$ be a Banach space. If $\mu\in\F_p(X)$ satisfies $\norm{\beta_{p,X}(\mu)}=1$, then
\begin{equation}\label{eq.2.98}
d_{\F_p(X)}\enpar{\mu,\Acal_{p,X}}^p
\leq\frac{2}{1-p} \enpar{\norm{\mu}^p-1}
\end{equation}
\end{lemma}

\begin{proof}
Fix $C>\norm{\mu}\ge 1$. By~\eqref{eq:atomic-formula} and the symmetry of $\Acal_{p,X}$, there are numbers $a_j>0$ and elementary molecules $\mu_j\in\Acal_{p,X}$ such that
\[
\mu =\sum_{j=1}^{\infty}a_j\, \mu_j
\]
and
\begin{equation}\label{eq:near-optimal-decomposition}
\sum_{j=1}^{\infty}a_j^p< C^p.
\end{equation}
Set, for convenience, $\beta=\beta_{p,X}$. Since $\beta(\mu_j)\in S_X$ for all $j\in\NN$,
\begin{equation}\label{eq:l1large}
1=\norm{\beta(\mu)}=\norm{\sum_{j=1}^\infty a_j\, \beta(\mu_j)}\le \sum_{j=1}^\infty a_j.
\end{equation}
Similarly, for $k\in\NN$,
\begin{equation}\label{eq:l1largebis}\begin{split}
\abs{1-a_k}
&=\Big\vert \norm{\beta(\mu)}-a_k \norm{\beta(\mu_k)}\Big\vert \\
&\le \norm{\beta(\mu) -a_k\beta(\mu_k) }=\norm{\sum_{j\not=k} a_j\, \beta(\mu_j)}\le \sum_{j\not=k}^\infty a_j.
\end{split}\end{equation}

By \eqref{eq:near-optimal-decomposition} and \eqref{eq:l1large}, applying \eqref{eq:sb} with $\varepsilon=C^{-1}$ yields
\[
1\leq \sum_{j=1}^\infty a_j \le C^{p/(1-p)} \sup_{j\in\NN} a_j.
\]
Since $\lim a_j=0$, there is $k\in\NN$ such that $a_k\ge a_j$ for all $j\in\NN$. The expansion
\[
\mu-\mu_k=(a_k-1)\mu_k +\sum_{j\not=k} a_j \mu_j
\]
and inequalities \eqref{eq:l1largebis} and \eqref{eq:lpl1} yield
\begin{align*}
\norm{\mu-\mu_k}^p &\le \abs{1-a_k}^p + \sum_{j\not=k} a_j^p
\le\enpar{\sum_{j\not=k} a_j}^p+ \sum_{j\not=k} a_j^p\\
&\le 2 \sum_{j\not=k} a_j^p \le 2 \enpar{C^p -a_k^p}
\le 2 \enpar{C^p - C^{-p^2/(1-p)}}.
\end{align*}
Letting $C$ tend to $\norm{\mu}$ gives
\begin{equation}\label{eq:2.99}
d_{\F_p(X)}(\mu,\Acal_{p,X})^p \le 2 g(\norm{\mu}^p),
\end{equation}
where
\[
g(t)=t-t^{-p/(1-p)},\qquad t\ge 1.
\]
Since $g(1)=0$, and for $t\ge 1$,
\[
g'(t)
=1+\frac{p}{1-p}t^{-1/(1-p)}
\leq 1+ \frac{p}{1-p} = \frac{1}{1-p},
\]
the mean value theorem yields $g(t) \le (t-1)/(1-p)$. Combining this inequality with \eqref{eq:2.99} gives \eqref{eq.2.98}.
\end{proof}

The next proposition makes precise how the estimate above enters the construction required in Proposition~\ref{prop:quotient-obstruction}. It shows that once a finitely supported vector $\gamma\in\ker\beta_{p,X}$ satisfies an appropriate uniform separation property, every sufficiently small perturbation of $\gamma$ has the distance property required there. Finding such a vector $\gamma$ will be the main task of the next section.

\begin{proposition}\label{prop:stable-kernel-direction}
Let $0<p<1$, let $X$ be a Banach space, and let $\gamma\in\spn(\delta_{p,X}(X))$ satisfy $\beta_{p,X}(\gamma)=0$ and
\[
c:=\inf_{\mu,\nu\in\Acal_{p,X}}
d_{\F_p(X)}\enpar{\gamma,[\mu-\nu]}>0.
\]
Then there exists $\varepsilon>0$ such that, whenever $\rho\in\F_p(X)$ satisfies $\norm{\rho-\gamma}<\varepsilon$, one has
\[
d_{\F_p(X)}\enpar{\mu,[\rho]}=1
\qquad\text{for every }\mu\in\Acal_{p,X}.
\]
In fact, one may take
$\varepsilon=c((1-p)/2)^{1/p}$.
\end{proposition}

\begin{proof}
Fix $\mu\in\Acal_{p,X}$. For $t\neq0$ and $\nu\in\Acal_{p,X}$,
\[
\norm{\mu+t\gamma-\nu}
=\abs{t}\norm{\gamma+t^{-1}(\mu-\nu)}
\geq c\abs{t},
\]
and hence $d_{\F_p(X)}(\mu+t\gamma,\Acal_{p,X})\geq c\abs{t}$. Moreover, since $\beta_{p,X}(\gamma)=0$, we have $\norm{\beta_{p,X}(\mu+t\gamma)}=1$. Hence Lemma~\ref{lem:concentration} yields
\[
d_{\F_p(X)}\enpar{\mu+t\gamma,\Acal_{p,X}}^p
\leq \frac{2}{1-p}
\enpar{\norm{\mu+t\gamma}^p-1}.
\]
Combining this with
$d_{\F_p(X)}(\mu+t\gamma,\Acal_{p,X})\geq c\abs{t}$, we obtain
\[
c^p\abs{t}^p
\leq \frac{2}{1-p} \enpar{\norm{\mu+t\gamma}^p-1},
\]
and therefore
\[
\norm{\mu+t\gamma}^p
\geq 1+\frac{1-p}{2}c^p\abs{t}^p,
\qquad t\in\RR.
\]

Let $\varepsilon=c((1-p)/2)^{1/p}$ and suppose that $\norm{\rho-\gamma}<\varepsilon$. By the reverse $p$-triangle
inequality, for every $t\in\RR$,
\begin{align*}
\norm{\mu+t\rho}^p
&\geq \norm{\mu+t\gamma}^p -\abs{t}^p\norm{\rho-\gamma}^p\\
&\geq 1+
\enpar{\frac{1-p}{2}c^p-\norm{\rho-\gamma}^p}\abs{t}^p
\geq1.
\end{align*}
Thus $d_{\F_p(X)}(\mu,[\rho])\geq1$, while the reverse inequality
follows by taking $t=0$. Since $\mu\in\Acal_{p,X}$ was arbitrary,
the conclusion follows.
\end{proof}
\section{Uniform separation from differences of elementary molecules}\label{sec:transverse}\noindent
Let $0<p\leq1$ and let $X$ be a Banach space. We shall construct a vector in $\F_p(X)$ which stays uniformly away from every line through the origin generated by a difference of two elementary molecules.

Choose $e\in S_X$ and
\[
L\geq 2+2^{1/p},
\]
and define
\begin{equation}\label{eq:z0}
\gamma=\delta_{p,X}(Le)+\delta_{p,X}(2Le)+\delta_{p,X}(3Le)+\delta_{p,X}(4Le)-2\delta_{p,X}(5Le).
\end{equation}
Notice that
\begin{equation}\label{tag3.7}
\beta_{p,X}(\gamma)=Le+2Le+3Le+4Le-2(5Le)=0.
\end{equation}

To establish the required separation property, we construct a suitable Lipschitz map. For $k=1,\ldots,5$, define
\[
h_k(x)=\max\enbrace{0,1-\norm{x-kLe}}, \qquad x\in X.
\]
Each function $h_k$ takes values in $[0,1]$ and is $1$-Lipschitz. Moreover, if
\[
U_k=\enbrace{ x\in X \colon \norm{x-kLe}<1}, \qquad k=1,\ldots,5,
\]
then $h_k(x)\neq0$ precisely when $x\in U_k$. Put
\begin{equation}\label{tag3.1}
h\colon X\to\ell_p^5, \qquad h(x)=(h_j(x))_{j=1}^5.
\end{equation}

\begin{lemma}\label{lem:h-lipschitz}
Let $X$ be a Banach space, $0<p\le 1$, $e\in S_X$ and $L\ge 2+2^{1/p}$. Then the map $h$ is $1$-Lipschitz. Moreover, if $(e_k)_{k=1}^5$ denotes the canonical basis of $\ell_p^5$,
\begin{enumerate}[label=(\roman*),leftmargin=*,widest=ii]
\item\label{it:h2} $h(0)=0$, $h(kLe)=e_k$ for $k=1$, \dots $5$, and
\item\label{it:h1} $h(x)$ has at most one nonzero coordinate for all $x\in X$.
\end{enumerate}
\end{lemma}

\begin{proof}
Since $L\geq1$, assertion \ref{it:h2} follows immediately from the definition of the functions $h_k$.

If $x\in U_k$ and $y\in U_j$ with $k\neq j$, then
\begin{equation}\label{tag3.2}\begin{split}
\norm{x-y}
&\geq \norm{(k-j)Le}-\norm{x-kLe}-\norm{y-jLe} \\
&> \abs{k-j}L-2
\geq L-2.
\end{split}\end{equation}
In particular, since $L>2$, taking $x=y$ shows that the sets $U_k$ are pairwise disjoint. Thus \ref{it:h1} holds.

Fix $x$, $y\in X$. Suppose that there are distinct $k$, $j\in\{1,\ldots,5\}$ such that $x\in U_k$ and $y\in U_j$. By \eqref{tag3.2},
\[
\norm{x-y}>L-2\geq2^{1/p}.
\]
Since $0\leq h_k(x)$, $h_j(y)\leq1$ and all the remaining coordinates vanish,
\[
\norm{h(x)-h(y)}_p^p
= h_k(x)^p+h_j(y)^p
\leq 2.
\]
Hence, $\norm{h(x)-h(y)}_p < \norm{x-y}$.

It remains to consider the case in which there is $k\in\{1,\ldots,5\}$ such that $h_j(x)=h_j(y)=0$ for all $j\in\{1,\dots,5\}$ with $j\not=k$. Since $h_k$ is $1$-Lipschitz,
\[
\norm{h(x)-h(y)}_p
=\abs{h_k(x)-h_k(y)}
\leq\norm{x-y}.
\]
Thus $h$ is $1$-Lipschitz.
\end{proof}

\begin{lemma}\label{lem:transverse}
Let $X$ be a Banach space, $0<p\le 1$, $e\in S_X$ and $L\ge 2+2^{1/p}$. Let $\gamma\in\F_p(X)$ be the vector defined in \eqref{eq:z0}. Then, for all $\mu$, $\nu\in\Acal_{p,X}$, and all $c\in\RR$,
\[
\norm{\gamma-c(\mu-\nu)}\geq1.
\]
\end{lemma}

\begin{proof}
Let $h$ be the map defined in \eqref{tag3.1}. By Lemma~\ref{lem:h-lipschitz} and the universal property of Lipschitz-free $p$-spaces, there is a linear contraction
\[
P:=E[h;p] \colon \F_p(X) \to \ell_p^5
\]
such that $P(\delta_{p,X}(x))=h(x)$ for all $x\in X$. Set
\[
(a_k)_{k=1}^5=P\enpar{\gamma-c(\mu-\nu)}=P(\gamma)-P(c(\mu-\nu)).
\]
By Lemma~\ref{lem:h-lipschitz}\ref{it:h1}, the vector $P(c(\mu-\nu))$ has at most four nonzero coordinates, whereas by Lemma~\ref{lem:h-lipschitz}\ref{it:h2},
\begin{equation*}
P(\gamma)=(1,1,1,1,-2).
\end{equation*}
Hence at least one coordinate $a_k$ satisfies $|a_k|\geq1 $. Consequently,
\[
1\le \norm{P\enpar{\gamma-c(\mu-\nu)}}\le \norm{\gamma-c(\mu-\nu)}.\qedhere
\]
\end{proof}
\section{The quotient counterexample}\label{sec:proof}\noindent
We now combine the preceding ingredients to prove Theorem~\ref{thm:main}. Proposition~\ref{prop:quotient-obstruction} already provides the required counterexample once a suitable vector $\rho$ has been found. To obtain the additional conclusion that the quotient space is isomorphic to $\F_p(X)$, we shall use the following elementary fact about hyperplanes (i.e., closed subspaces of co-dimension one) in quasi Banach spaces.

\begin{lemma}\label{lem:hyper}
Let $Y$ be a quasi-Banach space and $y\in Y$. Assume that there is $y^*\in Y^*$ such that $y^*(y)\not=0$. Then,
\begin{enumerate}[label=(\roman*),leftmargin=*,widest=ii]
\item\label{it:hyper:1} $Y/[y]$ is isomorphic to a hyperplane of $Y$, and
\item\label{it:hyper:2} all hyperplanes of $Y$ are isomorphic.
\end{enumerate}
\end{lemma}

\begin{proof}
We first recall that if $Z$ is a quasi-Banach space, $f\in Z^*$ is nonzero, and $z\in Z$ satisfies $f(z)=1$, then $Z=\ker(f)\oplus [z]$. Indeed, the corresponding projection onto $\ker(f)$ is given by $x\mapsto x-f(x)z$, and is continuous.

Put $H=\ker(y^*)$. Since $y^*(y)\neq0$, the preceding observation, applied with $z=y/y^*(y)$, gives $Y=H\oplus[y]$. Hence the quotient $Y/[y]$ is naturally isomorphic to $H$, proving~\ref{it:hyper:1}.

To prove~\ref{it:hyper:2}, let $M$ be any hyperplane of $Y$. There is nothing to prove if $M=H$, so assume that $M\neq H$, and write $M=\ker(f)$ for some nonzero $f\in Y^*$. Then both $\restr{y^*}{M}$ and $\restr{f}{H}$ are nonzero; otherwise one of the hyperplanes $M,H$ would be contained in the other, and hence they would coincide. Moreover,
\[
\ker\enpar{\restr{y^*}{M}}
=M\cap H
=\ker\enpar{\restr{f}{H}}.
\]
Applying the preceding observation first in $M$ and then in $H$, we obtain
\[
M\simeq (M\cap H)\oplus\RR
\simeq H.
\]
Thus every hyperplane of $Y$ is isomorphic to $H$, and consequently all hyperplanes of $Y$ are mutually isomorphic.
\end{proof}

\begin{proof}[Proof of Theorem~\ref{thm:main}]
Choose $e\in S_X$ and $L\geq 2+2^{1/p}$, and let $\gamma$ be defined by \eqref{eq:z0}. By Lemma~\ref{lem:transverse},
\[
c:=\inf_{\mu,\nu\in\Acal_{p,X}}
d_{\F_p(X)}\enpar{\gamma,[\mu-\nu]}\geq 1,
\]
while $\beta_{p,X}(\gamma)=0$. Set
\[
\rho=\gamma+\enpar{\frac{1-p}{4}}^{1/p}\delta_{p,X}(e).
\]
Then
\[
\norm{\rho-\gamma}
=\enpar{\frac{1-p}{4}}^{1/p}
<c\enpar{\frac{1-p}{2}}^{1/p},
\]
and hence Proposition~\ref{prop:stable-kernel-direction} gives
$d_{\F_p(X)}(\mu,[\rho])=1$ for every $\mu\in\Acal_{p,X}$.

Moreover, $\rho\in\spn(\delta_{p,X}(X))$ and
\[
\beta_{p,X}(\rho)
=\enpar{\frac{1-p}{4}}^{1/p}e\neq0.
\]
Therefore Proposition~\ref{prop:quotient-obstruction}, applied to $Y=\F_p(X)/[\rho]$, yields an isometric embedding $\varphi\colon X\to Y$ with $\varphi(0)=0$ and $[\varphi(X)]=Y$ such that there is no linear map $T\colon\spn(\varphi(X))\to X$ satisfying $T\circ\varphi=\Id_X$.

It remains to prove that $Y\simeq\F_p(X)$. Since $\beta_{p,X}(\rho)\neq0$, there is $x^*\in X^*$ such that $x^*(\beta_{p,X}(\rho))\neq0$. Thus $x^*\circ\beta_{p,X}\in\F_p(X)^*$ does not vanish at $\rho$, and Lemma~\ref{lem:hyper}\ref{it:hyper:1} shows that $Y$ is isomorphic to a hyperplane of $\F_p(X)$. On the other hand, $\F_p(X)\oplus\RR\simeq\F_p(X)$ by \cite{AACD2020b}*{Corollary~3.3}, so $\F_p(X)$ is isomorphic to one of its hyperplanes. Hence Lemma~\ref{lem:hyper}\ref{it:hyper:2} yields $Y\simeq\F_p(X)$.
\end{proof}
\section{The isometric \texorpdfstring{$p$}{p}-lifting property}\label{sec:isometric-p-liftings}\noindent
The quantitative proximity estimate from Section~\ref{sec:BVCL} also imposes a strong restriction on linear right inverses of the $p$-barycentric map. The additional ingredient is a simple two-dimensional observation about how elementary molecules can lie in a Lipschitz-free $p$-space.

We begin with a closedness property of the set of elementary molecules. For this purpose, recall that given $0<p\le 1$ and a metric space $\Mcal$, the universal property of $\F_p(\Mcal)$ yields a canonical linear contraction
\[
J_p[\Mcal] \colon \F_p(\Mcal) \to \F_1(\Mcal)
\]
such that
\begin{equation}\label{tag:99}
J_p[\Mcal]\circ\delta_{p,\Mcal}=\delta_{1,\Mcal}.
\end{equation}

\begin{proposition}\label{prop:elementary molecules-closed}
Let $0<p\leq1$ and let $\Mcal$ be a pointed complete metric space. Then $\Acal_{p,\Mcal}$ is closed in $\F_p(\Mcal)$.
\end{proposition}

\begin{proof}
Suppose first that $p=1$. By \cite{GLPPR2018}*{Proposition~2.9} the weak closure of $\Acal_{1,\Mcal}$ is contained in $ \Acal_{1,\Mcal}\cup\{0\}$. Hence the same is true for its norm closure. Since every elementary molecule has norm one, the origin cannot belong to the norm closure of $\Acal_{1,\Mcal}$. Thus $\Acal_{1,\Mcal}$ is norm closed.

Now let $0<p<1$. By \cite{AlbiacAnsorena2026b}*{Theorem~A}, the map $J_p[\Mcal]$ is one-to-one. Moreover, by \eqref{tag:99},
\[
J_p[\Mcal]\enpar{\Acal_{p,\Mcal}}=\Acal_{1,\Mcal}.
\]
Injectivity therefore gives that
\[
\Acal_{p,\Mcal} =J_p[\Mcal]^{-1} \enpar{\Acal_{1,\Mcal}}.
\]
Since $J_p[\Mcal]$ is continuous and $\Acal_{1,\Mcal}$ is closed, the conclusion follows.
\end{proof}

The next lemma is purely algebraic. It says that a two-dimensional subspace can contain elementary molecules associated with at most three unordered pairs of points.

Given a set $\Mcal$ and $k\in\NN$, we write $[\Mcal]^k$ for the collection of all subsets of $\Mcal$ with cardinality $k$.

\begin{lemma}\label{lem:elementary molecules-two-dimensional}
Let $0<p\le 1$, let $\Mcal$ be a $p$-metric space, and let $E$ be a two-dimensional subspace of $\F_p(\Mcal)$. Then the set
\[
\mathcal{D}=\enbrace{\{x,y\} \in [\Mcal]^2 \colon \mu_{p,\Mcal}(x,y) \in E}
\]
has at most three elements.
\end{lemma}

\begin{proof}
For distinct $a,b\in\Mcal$, put $\Delta(a,b)=\delta_{p,\Mcal}(a)-\delta_{p,\Mcal}(b)$. Since $\mu_{p,\Mcal}(a,b)$ is a nonzero scalar multiple of $\Delta(a,b)$, we have
\[
\{a,b\}\in\mathcal{D}
\quad\Longleftrightarrow\quad
\Delta(a,b)\in E.
\]

There is nothing to prove if $\#\mathcal{D}\leq1$. Otherwise, choose distinct $\{x,y\},\{u,v\}\in\mathcal{D}$, labeling the second pair so that $v\notin\{x,y\}$. The vectors $\Delta(x,y)$ and $\Delta(u,v)$ are linearly independent, and hence every $\mu\in E$ has a unique representation
\[
\mu=t\Delta(x,y)+s\Delta(u,v),\qquad t,s\in\RR.
\]

Suppose first that $u\notin\{x,y\}$. Then $x,y,u,v$ are pairwise distinct, and if $ts\neq0$, the above representation involves all four points. Hence it cannot be a scalar multiple of $\Delta(a,b)$ for any $a\neq b$. It follows that
\[
\mathcal{D}=\enbrace{\{x,y\},\{u,v\}}.
\]

It remains to consider the case $u\in\{x,y\}$. Relabeling $x,y$ if necessary, we may assume that $u=y$. Then
\[
t\Delta(x,y)+s\Delta(y,v)
=t\delta_{p,\Mcal}(x)+(s-t)\delta_{p,\Mcal}(y) -s\delta_{p,\Mcal}(v).
\]
Such a vector can be a scalar multiple of $\Delta(a,b)$ only if one of $t$, $s$, or $s-t$ vanishes. These three possibilities correspond respectively to the pairs $\{y,v\}$, $\{x,y\}$, and $\{x,v\}$. Thus
\[
\mathcal{D}\subset
\enbrace{\{x,y\},\{y,v\},\{x,v\}}.
\]
The reverse inclusion follows from the choice of the first two pairs and from $\Delta(x,v)=\Delta(x,y)+\Delta(y,v)\in E$. Hence $\#\mathcal D=3$ in this case.
\end{proof}

We now combine this observation with the quantitative proximity estimate from Section~\ref{sec:BVCL}.

\begin{corollary}\label{cor:barycenter-no-isometry}
Let $0<p<1$ and let $X$ be a Banach space. Then the $p$-barycentric map $\beta_{p,X}\colon\F_p(X)\to X$ cannot be isometric on any linear subspace of $\F_p(X)$ of dimension at least two.
\end{corollary}

\begin{proof}
It is enough to rule out a two-dimensional subspace. Suppose that $E\subseteq\F_p(X)$ is two-dimensional and that $\restr{\beta_{p,X}}{E}$ is an isometry. For every $u\in S_E$ we then have $\norm{\beta_{p,X}(u)}=1$. Lemma~\ref{lem:concentration} gives
\[
d_{\F_p(X)} \enpar{u,\Acal_{p,X}}^p
\leq \frac{2}{1-p} \enpar{ \norm{u}^p-1}
=0.
\]
Proposition~\ref{prop:elementary molecules-closed} implies that $ u\in\Acal_{p,X}$. Thus
\[
S_E\subseteq E\cap\Acal_{p,X}.
\]
Lemma~\ref{lem:elementary molecules-two-dimensional}, however, shows that there are at most three unordered pairs $\{x,y\}$ for which $\mu_{p,X}(x,y)\in E$. Each such pair determines exactly two elementary molecules, namely $\mu_{p,X}(x,y)$ and $\mu_{p,X}(y,x)$, and so,
\[
\#\enpar{E\cap\Acal_{p,X}}\leq 6.
\]
This is impossible, since the unit sphere of any two-dimensional space is infinite.
\end{proof}

There are, of course, infinite-dimensional $p$-Banach spaces with the isometric $p$-lifting property. In particular, every Lipschitz-free $p$-space has it. Indeed, fix a $p$-metric space $\Mcal$ and let
\[
Z=\F_p(\Mcal).
\]
The isometric map
\[
\delta_{p,Z}\circ\delta_{p,\Mcal}\colon \Mcal\to\F_p(Z)
\]
linearizes to a contractive linear map
\[
R\colon Z\to\F_p(Z)
\]
such that
\[
R\circ\delta_{p,\Mcal}
=\delta_{p,Z}\circ\delta_{p,\Mcal}.
\]
Applying the $p$-barycentric map associated with $Z$ gives
\[
\beta_{p,Z}\circ R\circ\delta_{p,\Mcal}
= \delta_{p,\Mcal}.
\]
Since the linear span of $\delta_{p,\Mcal}(\Mcal)$ is dense in $Z$, it follows that
\[
\beta_{p,Z}\circ R=\Id_Z.
\]
Thus $Z$ has the isometric $p$-lifting property. For Banach spaces, however, the picture is completely different.

\begin{theorem}\label{thm:p-lifting}
Let $0<p<1$. A real Banach space $X$ has the isometric $p$-lifting property if and only if $ \dim(X)\le 1$.
\end{theorem}

\begin{proof}
The cases when $\dim(X) \le 1$ are immediate. Indeed, if $X=\{0\}$ there is nothing to prove, while if $\dim(X)=1$, the mapping $R\colon X \to \F_p(X)$ defined from $e\in S_X$ by
\[
t\, e \mapsto t \,\delta_{p,X}(e), \qquad t\in\RR,
\]
is a linear isometry such that $\beta_{p,X}\circ R=\Id_X$.

Conversely, suppose that $R\colon X\to\F_p(X)$ is a linear contraction satisfying $ \beta_{p,X}\circ R=\Id_X$ and set $E=R(X)$. Since $\beta_{p,X}$ is also contractive and $\beta_{p,X}\circ R$ is an isometry, both $R$ and $\restr{\beta_{p,X}}{E}$ are isometries. If $\dim(X)\ge 2$, then $\dim (E)\ge2$, contradicting Corollary~\ref{cor:barycenter-no-isometry}. Therefore $\dim(X)\le 1$.
\end{proof}

\section{Statements and Declarations}
\subsection*{Conflict of Interest}
The authors declare that they have no conflict of interest.
\subsection*{Data Availability} Since no datasets were generated or analyzed during the current study, data sharing does not apply to this article.
\subsection*{Generative AI disclosure}
The authors acknowledge the use of ChatGPT (OpenAI, GPT-5.6 Sol) in the development of this paper. Through an iterative exchange guided by the authors, ChatGPT produced a preliminary draft. The authors subsequently analyzed the draft in detail, revised and improved the mathematical contents, considerably and significantly simplifying the original arguments, and constructed the paper in its present form. Therefore, all statements, proofs and references appearing in this manuscript have been carefully revised by the authors, who accept the sole responsibility for any remaining inaccuracies or mistakes.

\end{document}